\documentclass[12pt,A4]{amsart}
\usepackage{amsfonts,amsthm,amsmath}
\usepackage{amssymb}
\usepackage{tikz}
\usepackage[latin2]{inputenc}
\usepackage{t1enc}
\usepackage[mathscr]{eucal}
\usepackage{indentfirst}
\usepackage{graphicx}
\usepackage{graphics}
\usepackage{mathrsfs}
\usepackage{enumerate}
\usepackage[pagebackref]{hyperref}

\usepackage{cite}
\usepackage{xcolor}

\numberwithin{equation}{section}
\def\blue{\textcolor{blue}}
\def\red{\textcolor{red}}

\theoremstyle{plain}
\newtheorem{theorem}{Theorem}[section]
\newtheorem{lemma}[theorem]{Lemma}
\newtheorem{corollary}[theorem]{Corollary}

\newtheorem{proposition}[theorem]{Proposition}

\theoremstyle{definition}

\newtheorem{conjecture}[theorem]{Conjecture}
\newtheorem{remark}[theorem]{Remark}
\newtheorem{?}[theorem]{Problem}

\newcommand{\bs}{{\mathfrak S}}

\newcommand{\des}{{\rm des}}

\newcommand{\altdes}{{\rm altdes}}
\newcommand{\Altdes}{{\rm \widehat{D}}}
\newcommand{\altmaj}{{\rm altmaj}}

\newcommand{\msn}{\mathfrak{S}_n}

\newcommand{\ms}{\mathfrak{S}}
\newcommand{\msnn}{\mathfrak{S}_{n+1}}

\newcommand{\st}{\mathsf{st}}

\newcommand{\lrf}[1]{\lfloor #1\rfloor}

\newcommand{\maj}{{\rm maj}}
\newcommand{\bfa}{{\bf a}}
\newcommand{\bfb}{{\bf b}}
\newcommand{\bfc}{{\bf c}}
\newcommand{\bfd}{{\bf d}}

\newcommand\floor[1]{\lfloor#1\rfloor}
\usepackage{delimset}

\def\boxit#1{\leavevmode\hbox{\vrule\vtop{\vbox{\kern.33333pt\hrule\kern1pt\hbox{\kern1pt\vbox{#1}\kern1pt}}\kern1pt\hrule}\vrule}}

\usepackage{collectbox}


\begin{document}

\title[Alternating descent polynomials]{Positivity and divisibility  of alternating \\descent polynomials}

\author[Z. Lin]{Zhicong Lin}
\address[Zhicong Lin]{Research Center for Mathematics and Interdisciplinary Sciences, Shandong University, Qingdao 266237, P.R. China}
\email{linz@sdu.edu.cn}

\author[S.-M.~Ma]{Shi-Mei Ma}
\address[Shi-Mei Ma]{School of Mathematics and Statistics,
        Northeastern University at Qinhuangdao,
         Hebei 066004, P.R. China}
\email{shimeimapapers@163.com}

\author[D.G.L.~Wang]{David G.L. Wang}
\address[David G.L. Wang]{School of Mathematics and Statistics \& Beijing Key Laboratory on MCAACI, Beijing Institute of Technology, 102488 Beijing, R.R. China}
\email{glw@bit.edu.cn}

\author[L. Wang]{Liuquan Wang}
\address[Liuquan Wang]{School of Mathematics and Statistics, Wuhan University, Wuhan 430072, Hubei, People's Republic of China}
\email{
wanglq@whu.edu.cn}

\date{\today}

\begin{abstract}
The alternating descent statistic on permutations was introduced by Chebikin
as a variant of the descent statistic. We show that the alternating descent polynomials on permutations are unimodal via a five-term recurrence relation. We also found a  quadratic  recursion for the alternating major index $q$-analog of the alternating descent polynomials. As an interesting application of this  quadratic recursion, we show that $(1+q)^{\lfloor n/2\rfloor}$ divides $\sum_{\pi\in\bs_n}q^{\altmaj(\pi)}$, where $\bs_n$ is the set of all permutations of $\{1,2,\ldots,n\}$ and $\altmaj(\pi)$ is the alternating major index of $\pi$. This leads us to discover a $q$-analog of $n!=2^{\ell}m$, $m$ odd, using the statistic of alternating major index. Moreover, we study  the $\gamma$-vectors of the alternating descent polynomials by using these two recursions and the {\bf cd}-index. Further intriguing conjectures are formulated, which indicate that the alternating descent statistic deserves more work.
\end{abstract}



\maketitle


\section{Introduction}

Let $\msn$ denote the symmetric group of all permutations of $[n]$, where $[n]=\{1,2,\ldots,n\}$.
Here a permutation  $\pi$ in $\msn$ is viewed as a word $\pi=\pi_1\pi_2\cdots\pi_n$ with $\pi_i=\pi(i)$. A {\it descent} of a permutation $\pi\in\msn$ is an index $i\in[n-1]$ such that $\pi_i>\pi_{i+1}$.
The classical {\it Eulerian polynomial} $A_n(t)$ may be defined as the descent polynomial on $\bs_n$ (cf.~\cite{comtet,fs,pe2}):
$$
A_n(t):=\sum_{\pi\in\msn}t^{\des(\pi)},
$$
where $\des(\pi)$ is the number of descents of $\pi$.

As a variation of the descent statistic, the number of {\em alternating descents}  of a permutation $\pi\in\msn$, denoted $\altdes(\pi)$,  is defined as the cardinality of
$$\Altdes(\pi):=\{2i: \pi_{2i}<\pi_{2i+1}\}\cup \{2i+1: \pi_{2i+1}>\pi_{2i+2}\}.$$
We say that $\pi$ has a {\it 3-descent} at index $i$ if $\pi_i\pi_{i+1}\pi_{i+2}$ has one of the patterns: $132$, $213$ or $321$. These two permutation statistics were introduced by Chebikin in~\cite{Chebikin08}, where he showed that
 the alternating descents on $\msn$ is equidistributed
with the 3-descent statistic on $\{\pi\in\msnn : \pi_1=1\}$.
The equations
$$\widehat{A}_n(t)=\sum_{\pi\in\msn}t^{\altdes(\pi)}=\sum_{k=0}^{n-1}\widehat{A}_{n,k}t^k$$
define the {\it alternating Eulerian polynomials} $\widehat{A}_n(t)$ and the {\it alternating Eulerian numbers} $\widehat{A}_{n,k}$.
The first few $\widehat{A}_n(t)$ are given as follows:
\begin{align*}
   \widehat{A}_1(t)& =1,\quad
  \widehat{A}_2(t)=1+t, \quad
  \widehat{A}_3(t)=2+2t+2t^2, \\
  \widehat{A}_4(t)& =5+7t+7t^2+5t^3,\quad\widehat{A}_5(t)=16+26t+36t^2+26t^3+16t^4.
\end{align*}
Note that $\widehat{A}_{n,n-1}$ is the famous  {\it Euler number} (see~\cite{andre,Stanley}) of order $n$, usually denoted by $E_n$, which enumerates the permutations $\pi\in\msn$ that is {\it down-up}:
$$\pi_1>\pi_2<\pi_3>\pi_4<\cdots.$$
Chebikin~\cite[Lemma~6.1]{Chebikin08} obtained the following recurrence relation for the alternating Eulerian numbers:
\begin{equation}\label{rec:che}
\sum_{i=0}^n\sum_{j=0}^k\binom{n}{i}\widehat{A}_{i,j}\widehat{A}_{n-i,k-j}=(n+1-k)\widehat{A}_{n,k}+(k+1)\widehat{A}_{n,k+1},
\end{equation}
from which he computed the exponential generating function for $\widehat{A}_n(t)$
\begin{equation}\label{EGF-Anx}
1+\sum_{n\geq 1}t\widehat{A}_{n}(t)\frac{z^n}{n!}=\frac{1-t}{1-t(\sec(1-t)z+\tan(1-t)z)}.
\end{equation}

In recent years, the alternating descents and its associated permutation statistics have attracted the attention of several authors~\cite{Gessel14,my,rem}. In particular, Remmel~\cite{rem} calculated the joint distribution of the statistics of alternating descents and alternating major index.  Gessel and Zhuang~\cite{Gessel14} showed that  permutations with even valleys and odd peaks can be characterized by alternating descents/runs,  which leads to a solution of  a  question arising in combinatorial topology posed by Nicolaescu. Ma and Yeh~\cite{my} expressed the polynomial $\widehat{A}_n(t)$ in terms of the {\em derivative polynomial} $P_n(t)$, where $P_n(\tan\theta)=\frac{d^n}{d\theta^n}\tan\theta$.

In this paper, we investigate the positivity and divisibility properties  of the alternating descent polynomials. Some further conjectures, which indicate that the alternating descent statistic deserves more work, will also be presented in the last section of this paper.

\subsection{Main results}
A polynomial $h(t)=\sum_{i=0}^n h_it^i$ with real coefficients is said to be
\begin{itemize}
\item {\em unimodal}  if there exists $0\leq c\leq n$, such that $h_0\leq h_1\leq\cdots h_c\geq h_{c+1}\geq\cdots h_n$;
\item or {\em palindromic} (of center $n/2$) if $h_i=h_{n-i}$ for all $0\leq i\leq \lfloor n/2\rfloor$.
\end{itemize}
Palindromic and unimodal polynomials with combinatorial meanings arise frequently in algebra, combinatorics and geometry (cf.~\cite{Ath,bran}).
One of the most interesting example is the Eulerian polynomials   (cf.~\cite[p.~292]{comtet}), which even possess the stronger property of $\gamma$-positivity~\cite{fs,pe2}. Recall that  palindromic polynomials in $\mathbb{R}[t]$ with center $n/2$  has a basis
$$
B_n:=\{t^k(1+t)^{n-2k}\}_{k=0}^{\lfloor n/2\rfloor}.
$$
If $h(t)=\sum_{k=0}^{\lfloor n/2\rfloor}\gamma_kt^k(1+t)^{n-2k}$, we call $\{\gamma_k\}_{k=0}^{\lfloor n/2\rfloor}$ the {\em $\gamma$-vector} of $h$. If the $\gamma$-vector of $h$ is nonnegative, then we say $h$ is $\gamma$-positive. Since each term $t^k(1+t)^{n-2k}$ in the expansion is palindromic and unimodal, $\gamma$-positivity implies unimodality.

Besides the Eulerian polynomials, many other palindromic and unimodal polynomials arising in the study of permutation statistics  are known to be $\gamma$-positive~\cite{pe2}.  However, it turns out that the alternating Eulerian polynomial  $\widehat{A}_n(t)$ is  unimodal and has $\gamma$-vector  alternates in sign.
\begin{theorem}\label{main:1}
For $n\geq1$, the alternating Eulerian polynomial $\widehat{A}_n(t)$ is palindromic and unimodal. Moreover, if we write
$$
\widehat{A}_n(t)=\sum_{k=0}^{\lrf{({n-1})/{2}}}a(n,k)(-2t)^{k}(1+t)^{n-1-2k} \text{ and }
a_n(x)=\sum_{k=0}^{\lrf{({n-1})/{2}}}a(n,k)x^k,
$$
then $a_n(x-1)$ is the descent polynomial over Simsun permutations of length $n-1$.
\end{theorem}

Let $\maj(\pi):=\sum_{\pi_i>\pi_{i+1}}i$ and $\altmaj(\pi):=\sum_{i\in\Altdes(\pi)}i$ be the {\em major index} and the {\em alternating major index} of $\pi\in\ms_n$, respectively. The major index is a {\em Mahonian statistic}, that is
\begin{equation}\label{mahon}
\sum_{\pi\in\ms_n}q^{\maj(\pi)}=[n]_q!,
\end{equation}
 where $[n]_q:=1+q+\cdots q^{n-1}$ and $[n]_q!:=\prod_{j=1}^n[j]_q$ is the $q$-analog of $n!$.
 The alternating major index $q$-analog of $\widehat{A}_n(t)$ was first considered by Remmel~\cite{rem}:
$$
\widehat{A}_n(t,q):=\sum_{\pi\in\ms_n}t^{\altdes(\pi)}q^{\altmaj(\pi)}.
$$
He computed the exponential generating function for $\widehat{A}_n(t,q)$ via a homomorphism of the ring of symmetric functions that
\begin{equation}\label{gen:rem}
\sum_{n\geq0}\frac{z^n}{n!}\frac{\widehat{A}_n(t,q)}{(t;q)_{n+1}}=\sum_{k\geq0}t^k\prod_{j=0}^k(\sec(zq^j)+\tan(zq^j)),
\end{equation}
which is a $q$-analog of~\eqref{EGF-Anx}. Here $(t;q)_n:=\prod_{i=0}^{n-1}(1-tq^i)$ is the $q$-shifted factorial.
 Our second result indicates that the generating function
 $\widehat{A}_n(1,q)$ of alternating major index over $\ms_n$ provides a $q$-analog of $n!=2^{\ell}m$ with $m$ odd.



\begin{theorem}
\label{conj:altmaj}
For $n\ge 1$, we have
\begin{equation}\label{fac:altmaj}
\widehat{A}_n(1,q)=\sum_{\pi\in\mathfrak{S}_n}q^{\altmaj(\pi)}
=\widehat{E}_n(q)\prod_{k=1}^{\lfloor\log_2 n\rfloor}\prod_{i=1}^{\lfloor n/2^k \rfloor}(1+q^i),
\end{equation}
where $\widehat{E}_n(q)$ is a palindromic polynomial in $\mathbb{Z}[q]$
with constant term the Euler number~$E_n$.
\end{theorem}
 For the sake of convenience, we list the first few factorization (in the ring $\mathbb{Q}[q]$)  of $\widehat{A}_n(1,q)$ in the following:
\begin{align*}
\widehat{A}_2(1,q)&=1+q,\\
\widehat{A}_3(1,q)&=(1+q)(2-q+2q^2),\\
\widehat{A}_4(1,q)&=(1+q)^2(1+q^2)(5-7q+5q^2),\\
\widehat{A}_5(1,q)&=(1+q)^2(1+q^2)(16-23q+18q^2-7q^3+18q^4-23q^5+16q^6),\\
\widehat{A}_6(1,q)&=(1+q)^2(1+q^2)(1+q^3)\widehat{E}_6(q),\\
\widehat{A}_7(1,q)&=(1+q)^2(1+q^2)(1+q^3)\widehat{E}_7(q),\\
\widehat{A}_8(1,q)&=(1+q)^3(1+q^2)^2(1+q^3)(1+q^4)\widehat{E}_8(q),
\end{align*}
where
\begin{align*}
\widehat{E}_6(q)&=61-87q+66q^2-82q^3+129q^4-82q^5+66q^6-87q^7+61q^8,\\
\widehat{E}_7(q)&=272-389q+298q^2-375q^3+603q^4-497q^5+617q^6-743q^7\\
&\quad+617q^8-497q^9+603q^{10}-375q^{11}+298q^{12}-389q^{13}+272q^{14},\\
\widehat{E}_8(q)&=1385-3364q+3490q^2-3406q^3+4915q^4-5397q^5+4873q^6-4677q^7\\
&\quad+4873q^8-5397q^9+4915q^{10}-3406q^{11}+3490q^{12}-3364q^{13}+1385q^{14}.
\end{align*}

The rest of this paper is laid out as follows. In Section~\ref{Section-2}, we prove combinatorially two recurrence relations, a five-term recurrence relation and a  quadratic  recursion, for the alternating descent polynomials. Utilizing these two recursions and the {\bf cd}-index of symmetric groups, a proof of Theorem~\ref{main:1} is provided in Section~\ref{Sec:3}. As an application of the famous Fa\`{a} di Bruno's formula, proofs of Theorem~\ref{conj:altmaj} and its equivalent form (see Theorem~\ref{div:altmaj}) are given in Section~\ref{Sec:4}. Finally in Section~\ref{Sec:final}, we conclude the paper with some remarks and further conjectures.

\section{Two recurrence relations}\label{Section-2}

\subsection{A five-term recurrence relation}
It is well known that the Eulerian numbers $A_{n,k}$ defined by $A_n(t)=\sum_{k=0}^{n-1}A_{n,k}t^k$ satisfy  the three-term recurrence relation (cf.~\cite[Sec.~1.4]{pe2})
\begin{equation}\label{rec:euler}
A_{n+1,k}=(k+1)A_{n,k}+(n-k+1)A_{n,k-1}.
\end{equation}
However,  recursion~\eqref{rec:che} for the alternating Eulerian numbers $\widehat{A}_{n,k}$ derived by Chebikin is somewhat complicated. In order to show the unimodality of $\widehat{A}(t)$, we find the following five-term recursion for $\widehat{A}_{n,k}$, which is an analog of~\eqref{rec:euler}.
\begin{theorem}\label{new:rec}
For $n\geq 1$, the numbers $\widehat{A}_{n,k}$ satisfy
\begin{equation}\label{recuAnk}
2\widehat{A}_{n+1,k}=(k+1)(\widehat{A}_{n,k+1}+\widehat{A}_{n,k-1})+(n-k+1)(\widehat{A}_{n,k}+\widehat{A}_{n,k-2})
\end{equation}
with initial conditions $\widehat{A}_{1,0}=1$ and $\widehat{A}_{1,k}=0$ for $k\neq 0$.
\end{theorem}
Before we present the proof of Theorem~\ref{new:rec}, we need to introduce an involution $\mathcal{C}$ (the complement) on words with distinct letters from $\mathbb{N}$. For a word $w=w_1w_2\cdots w_n$  of length $n$ with distinct letters from $\mathbb{N}$, let
$\mathcal{C}(w)=w'_1w'_2\cdots w'_n$ be the  word such that if $w_i$ is the $l_i$-th largest letter in $w$ then $w'_i$ is the $l_i$-th smallest letter in $w$. We also define the \emph{normalization} of $w$, denoted $\mathcal{N}(w)$, to be the permutation in $\bs_n$ obtained from $w$ by replacing the $i$-th ($1\leq i\leq n$) smallest letter by $i$.
For example, if $w=36752$ then $\mathcal{C}(w)=63257$ and $\mathcal{N}(w)=24531$.
\begin{proof}[{\bf Proof of Theorem~\ref{new:rec}}]
We describe two ways of inserting that are called  \emph{min-inserting} and  \emph{max-inserting}. For each permutation $\pi\in\bs_n$ and $0\leq j\leq n$, we call two permutations in $\bs_{n+1}$,
 $$
\mathcal{N}( \pi_1\pi_2\cdots\pi_j\,0\,\mathcal{C}(\pi_{j+1}\pi_{j+2}\cdots\pi_n))\quad\text{and}\quad  \pi_1\pi_2\cdots\pi_j\, (n+1)\,\mathcal{C}(\pi_{j+1}\pi_{j+2}\cdots\pi_n),
 $$
 the min-inserting $1$ and max-inserting $n+1$ in the $j$-th space of $\pi$, respectively. For instance, the min-inserting of $1$ and the max-inserting of $6$ in the second space of $24315\in\bs_5$ are respectively
$$
\mathcal{N}(240\mathcal{C}(315))=\mathcal{N}(240351)=351462
 \quad\text{and}\quad
 246\mathcal{C}(315)=246351.
 $$

Recall that $\widehat{A}_{n+1,k}$ counts the number of permutations in $\ms_{n+1}$ with $k$ alternating descents. For a permutation $\pi\in\bs_n$, it is routine to check the following facts in four different cases:
\begin{itemize}
\item[(1)] The min-inserting of $1$ in  the even alternating descent space or max-inserting of $(n+1)$ in the odd alternating descent space will decrease the  number of alternating descents by $1$. Therefore, if $\altdes(\pi)=k+1$, then there are $k+1$ ways to inserting $1$ or $n+1$ to obtain a permutation in $\ms_{n+1}$ with $k$ alternating descents.
\item[(2)] The min-inserting of $1$ in the odd alternating descent space or max-inserting of $(n+1)$ in the even alternating descent space will increase the  number of alternating descents by $1$. Also, the max-inserting of $(n+1)$ in the $0$-th space or when $n$ is odd (resp. even), the min-inserting of $1$ (resp. the max-inserting of $n+1$) in the $n$-th space will both increase the  number of alternating descents by $1$. Therefore, if $\altdes(\pi)=k-1$, then there are $k-1+2=k+1$ ways to inserting $1$ or $n+1$ to obtain a permutation in $\ms_{n+1}$ with $k$ alternating descents.
\item[(3)] The min-inserting of $1$ in  the odd alternating ascent space or max-inserting of $(n+1)$ in the even alternating ascent space will increase the  number of alternating descents by $2$. Therefore, if $\altdes(\pi)=k-2$, then there are $n-1-(k-2)=n-k+1$ ways to inserting $1$ or $n+1$ to obtain a permutation in $\ms_{n+1}$ with $k$ alternating descents.
\item[(4)]  The min-inserting of $1$ in the even alternating ascent space or max-inserting of $n+1$ in the odd alternating ascent space will preserve the  number of alternating descents. Also, the max-inserting of $1$ in the $0$-th space or when $n$ is even (resp. odd), the min-inserting of $1$ (resp. the max-inserting of $n+1$) in the $n$-th space will both preserve the  number of alternating descents by $1$. Therefore, if $\altdes(\pi)=k$, then there are $n-1-k+2=n-k+1$ ways to inserting $1$ or $n+1$ to obtain a permutation in $\ms_{n+1}$ with $k$ alternating descents.
\end{itemize}

Summarizing all the above four cases, we obtain~\eqref{recuAnk}, since every permutation in $\bs_{n+1}$ will be constructed exactly two times by min-inserting $1$ or max-inserting $n+1$.
\end{proof}

It is possible to calculate the generating function formula~\eqref{EGF-Anx}
$$
\widehat{A}(t;z):=1+\sum_{n\geq1}t\widehat{A}_n(t)\frac{z^n}{n!}
$$
from Theorem~\ref{new:rec} as follows.
Multiplying both sides of~\eqref{recuAnk} by $t^{k+1}z^n/n!$ and summing over $k\geq0$ and $n\geq1$, we obtain
$$\left(1-\frac{(1+t^2)z}{2}\right)\frac{d}{dz}\widehat{A}(t;z)-\frac{1}{2}(1-t)(1+t^2)\frac{d}{dt}\widehat{A}(t;z)
=\frac{(t-1)^2}{2t}+\frac{t^2+2t-1}{2t}\widehat{A}(t;z)$$
after some manipulation.
Then there exists a function $f$ such that the solution of this partial differential equation can be written as
$$\widehat{A}(t;z/(1-t))=\frac{1-t}{1+t^2}(1-tf(-z-2\arctan t)).$$
Note that $\widehat{A}(t,0)=1$, and so $f(-2\arctan t)=-(1+t)/(1-t)$, which is equivalent to $f(t)=\frac{\tan(t/2)-1}{\tan(t/2)+1}$. Hence,
$$\widehat{A}(t;z/(1-t))=\frac{1-t}{1+t^2}\left(1+t\frac{1+\tan(z/2+\arctan t)}{1-\tan(z/2+\arctan t)}\right),$$
which, by the identity $\tan(\alpha+\beta)=\frac{\tan\alpha+\tan\beta}{1-\tan\alpha\tan\beta}$, implies that
\begin{align*}
\widehat{A}(t;z/(1-t))
&=(1-t)\left(\frac{1-\tan(z/2)}{1-t-(1+t)\tan(z/2)}\right)\\
&=\frac{1-t}{1-t\frac{1+\tan(z/2)}{1-\tan(z/2)}}=\frac{1-t}{1-t(\sec z+\tan z)},
\end{align*}
as desired.

\subsection{A  quadratic  recursion}
The classical {\em Euler--Mahonian polynomials} $A_n(t,q)$ introduced by Carlitz~\cite{car} in 1954 are  the $(\maj, \des)$-$q$-Eulerian polynomials
$$
A_n(t,q):=\sum_{\pi\in\bs_n}t^{\des(\pi)}q^{\maj(\pi)}.
$$
With $A_n(t,q)=\sum_{k=0}^{n-1}A_{n,k}(q)t^k$, Carlitz~\cite{car2} proved that the coefficients $A_{n,k}(q)$ satisfy the recurrence
$$
A_{n+1,k}(q)=[k+1]_qA_{n,k}(q)+q^k[n-k+1]_qA_{n,k-1}(q),
$$
which is a $q$-analogue of~\eqref{rec:euler}.
We are unable to derive a similar recurrence formula for $\widehat{A}_n(t,q)$ that is a $q$-analogue of
~\eqref{recuAnk}.
But instead we can show combinatorially the following recursion for $\widehat{A}_n(t,q)$, which is similar to a  quadratic  recursion for $A_n(t,q)$ derived by Park in~\cite[Corollary 3.6]{park}.

\begin{theorem}The polynomials $\widehat{A}_n(t,q)$ satisfy  the quadratic recursion
\begin{multline}\label{rec:2}
2\widehat{A}_{n+1}(t,q)=(1+tq)\widehat{A}_n(tq,q)+(1+tq^n)\widehat{A}_n(t,q)+\\
+\sum_{i=1}^{n-1}(1+t^2q^{2i+1}){n\choose i}\widehat{A}_i(t,q)\widehat{A}_{n-i}(tq^{i+1},q),
\end{multline}
where $n\geq1$ and $\widehat{A}_1(t,q)=1$.
\end{theorem}
\begin{proof}
For each $0\leq i\leq n$, consider two kinds of restricted alternating $q$-Eulerian polynomials:
$$
L_{n+1,i}(t,q):=\sum_{\pi\in\msnn\atop\pi(i+1)=n+1}t^{\altdes(\pi)}q^{\altmaj(\pi)}\quad\text{and}\quad S_{n+1,i}(t,q):=\sum_{\pi\in\msnn\atop\pi(i+1)=1}t^{\altdes(\pi)}q^{\altmaj(\pi)}.
$$
In view of the involution $\mathcal{C}$, we have
\begin{align*}
&L_{n+1,0}(t,q)=tq\widehat{A}_n(tq,q),\,\, L_{n+1,n}(t,q)=\widehat{A}_n(t,q),\\
 &S_{n+1,0}(t,q)=\widehat{A}_n(tq,q),\,\, S_{n+1,n}(t,q)=tq^n\widehat{A}_n(t,q)
\end{align*}
when $n+1$ is even and
\begin{align*}
&L_{n+1,0}(t,q)=tq\widehat{A}_n(tq,q),\,\,L_{n+1,n}(t,q)=tq^n\widehat{A}_n(t,q),\\
& S_{n+1,0}(t,q)=\widehat{A}_n(tq,q),\,\,S_{n+1,n}(t,q)=\widehat{A}_n(t,q)
\end{align*}
 when $n+1$ is odd.
Thus,
\begin{multline}\label{eq:twosides}
L_{n+1,0}(t,q)+L_{n+1,n}(t,q)+S_{n+1,0}(t,q)+S_{n+1,n}(t,q)\\=(1+tq)\widehat{A}_n(tq,q)+(1+tq^n)\widehat{A}_n(t,q),
\end{multline}
no matter when $n+1$ is even or odd.

Let $1\leq i\leq n-1$. There is a natural bijection $f$ between
 the set of permutations $\{\pi\in\msnn :\pi(i+1)=n+1\}$ and the set of triples
 $$\{(S,\pi',\pi'') : S\subseteq[n], |S|=i, \pi'\in\ms_{i}, \pi''\in\ms_{n-i-1}\}$$ defined as
 $$
 f(\pi)=(\{\pi_1,\pi_2,\ldots,\pi_{i}\},\mathcal{N}(\pi_1\pi_2\cdots\pi_{i}),\mathcal{N}(\mathcal{C}(\pi_{i+2}\pi_{i+3}\cdots\pi_{n+1}))),
 $$
 where $\mathcal{N}$ and $\mathcal{C}$ are the two operators introduced in the proof of Theorem~\ref{new:rec}. Moreover, this bijection satisfies that
 \begin{align*}
 \altdes(\pi)&=\altdes(\pi')+\altdes(\pi'')+2\cdot\chi(\text{$i$ even}),\\
 \altmaj(\pi)&=\altmaj(\pi')+\altmaj(\pi'')+(i+1)\cdot\altdes(\pi'')+ (2i+1)\cdot\chi(\text{$i$ even}),
 \end{align*}
 where $\chi(\mathsf{S})$ equals $1$, if the statement $\mathsf{S}$ is true; and $0$, otherwise.
 It follows that
 \begin{equation}\label{eq:L}
  L_{n+1,i}(t,q)=
 \begin{cases}
\,\,t^2q^{2i+1}{n\choose i}\widehat{A}_{i}(t,q)\widehat{A}_{n-i}(tq^{i+1},q)\qquad&\text{if $i$ is even;}\\
 \,\,{n\choose i}\widehat{A}_{i}(t,q)\widehat{A}_{n-i}(tq^{i+1},q)\qquad&\text{if $i$ is odd.}
 \end{cases}
 \end{equation}
 Similarly, we can show that
  \begin{equation}\label{eq:S}
 S_{n+1,i}(t,q)=
  \begin{cases}
\,\,t^2q^{2i+1}{n\choose i}\widehat{A}_{i}(t,q)\widehat{A}_{n-i}(tq^{i+1},q)\qquad&\text{if $i$ is odd;}\\
 \,\,{n\choose i}\widehat{A}_{i}(t,q)\widehat{A}_{n-i}(tq^{i+1},q)\qquad&\text{if $i$ is even.}
 \end{cases}
 \end{equation}

  Combining~\eqref{eq:twosides},~\eqref{eq:L} and~\eqref{eq:S} with
 $$
 2\widehat{A}_{n+1}(t,q)=\sum_{i=0}^{n} L_{n+1,i}(t,q)+\sum_{i=0}^{n} S_{n+1,i}(t,q)
 $$
 give~\eqref{rec:2},
 which completes the proof of the theorem.
\end{proof}

Applications of this  quadratic recursion will be given in the next two sections, including an elementary proof  that $(1+q)^{\lfloor n/2\rfloor}$ divides $\widehat{A}_n(1,q)$, which is a special case of Theorem~\ref{conj:altmaj}.

\section{Positivity properties}\label{Sec:3}

This section is devoted to the proof of Theorem~\ref{main:1} and some of its interesting consequences.

\begin{theorem}[First part of Theorem~\ref{main:1}]\label{unimodal}
The alternating Eulerian polynomial $\widehat{A}_n(t)$ is palindromic and unimodal for any $n\geq1$.
\end{theorem}
\begin{proof}
The involution $\mathcal{C}: \ms_n\rightarrow\ms_n$
 shows that $\widehat{A}_n(t)$ is palindromic. The unimodality of $\widehat{A}_n(t)$ is an easy consequence of Theorem~\ref{new:rec}.
We will prove this  by induction on $n$. By recurrence~\eqref{recuAnk}, for $0\leq k\leq \lfloor n/2\rfloor$, we have
\begin{align*}
&2(\widehat{A}_{n+1,k}-\widehat{A}_{n+1,k-1})\\
=&(n+1-2k)(\widehat{A}_{n,k}+\widehat{A}_{n,k-2})\\
&-(n+1-2k)\widehat{A}_{n,k-1}+(k+1)\widehat{A}_{n,k+1}-(n-k+2)\widehat{A}_{n,k-3}\\
\geq&(n+1-2k)(\widehat{A}_{n,k}+\widehat{A}_{n,k-2}-\widehat{A}_{n,k-1}-\widehat{A}_{n,k-3})
\geq0,
\end{align*}
where the inequalities follow from the induction hypothesis. This completes the proof of the theorem by induction.
\end{proof}

An index $i$, $2\leq i\leq n-1$, is a {\em double descent} of $\pi\in\ms_n$ if $\pi_{i-1}>\pi_i>\pi_{i+1}$. As introduced by Simion and Sundaram~\cite{sun}, a permutation in $\ms_n$ is called a {\em Simsun permutation} if it has no double descents, even after removing $n,n-1,\ldots,k$ for any $k$.
Let
$$
R(x):=\sum_{\pi\in RS_n} x^{\des(\pi)},
$$
where $RS_n$ is the set of all Simsun permutations in $\ms_n$. Chow and Shiu~\cite{CS} showed that $R(x)$ satisfy the recurrence formula
\begin{equation}\label{rec:sim}
R_n(x)=((n-1)x+1)R_{n-1}(x)+x(1-2x)R'_{n-1}(x).
\end{equation}
We have the following relationship between the $\gamma$-polynomial of $\widehat{A}_n(t)$ and $R_n(x)$.

\begin{theorem}[Second part of Theorem~\ref{main:1}]\label{th:sim} For $n\geq1$,  $a_n(x)=R_{n-1}(x+1)$.
\end{theorem}

\begin{proof} It follows from~\eqref{recuAnk} that
\begin{equation}\label{eq:a1}
2\widehat{A}_{n+1}(t)=(2t+n+1+nt^2-t^2)\widehat{A}_n(t)+(1-t)(1+t^2)\widehat{A}'_n(t).
\end{equation}
By the definition of $a_n(x)$, we have
\begin{equation}\label{eq:a2}
\widehat{A}_{n}(t)=(1+t)^{n-1}a_n(x),
\end{equation}
where $x=\frac{-2t}{(1+t)^2}$. Differentiating both sides yields
\begin{equation}\label{eq:a3}
\widehat{A}'_{n}(t)=(n-1)(1+t)^{n-2}a_n(x)+2(t-1)(1+t)^{n-4}\frac{d}{dx}a_n(x).
\end{equation}
Substituting~\eqref{eq:a2} and~\eqref{eq:a3} into~\eqref{eq:a1}, we get
\begin{align*}
a_{n+1}(x)=\biggl(\frac{nt^2+n+2t}{(1+t)^2}\biggr)a_n(x)-\frac{(1-t)^2(1+t^2)}{(1+t)^4}\frac{d}{dx}a_n(x).
\end{align*}
Since $\frac{nt^2+n+2t}{(1+t)^2}=n+(n-1)x$ and $\frac{(1-t)^2(1+t^2)}{(1+t)^4}=(1+x)(1+2x)$, we have
\begin{equation}\label{rec:gam}
a_{n+1}(x)=(n+(n-1)x)a_n(x)-(1+x)(1+2x)a_n'(x).
\end{equation}
Comparing with recursion~\eqref{rec:sim} for $R_n(x)$, we conclude $a_{n+1}(x)=R_n(x+1)$, as desired.
\end{proof}

\begin{remark}
Recursion~\eqref{rec:gam} is equivalent to
$$
a(n+1,k)=-(k+1)a(n,k+1)+(n-3k)a(n,k)+(n-2k+1)a(n,k-1)
$$
and so we have $a(n,1)=nE_n-E_{n+1}$. Note that the numbers $a(n,0)+a(n,1)$ appear as A034428 in~\cite{oeis}. Surprisingly, the nonnegativity of $a(n,k)$ does not follow directly from the above recursion.
\end{remark}

Theorem~\ref{th:sim} can also be proved by the {\bf cd}-index of $\ms_n$. For a subset $S\subseteq[n-1]$, define the monomial $u_S=u_1u_2\cdots u_{n-1}$ in two non-commuting variables $\bfa$ and $\bfb$ by:
$$
u_i=
\begin{cases}
\,\,\bfa \qquad\text{if $i\notin S$,}\\
\,\,\bfb \qquad\text{if $i\in S$.}
\end{cases}
$$
Consider the {\bf ab}-index of $\ms_n$ with respect to descent set statistic
$$
\Psi_n(\bfa,\bfb)=\sum_{\pi\in\ms_n} u_{D(\pi)},
$$
where $D(\pi)$ is the set of descents of $\pi$.

Let $SS_n:=\{\pi\in RS_n: \pi(n)=n\}$. For a Simsun permutation $\pi\in SS_n$, define the monomial $cd(\pi)$ in non-commuting variables $\bfc,\bfd$ as follows:  write out the monomial $u_{D(\pi)}$, and then replace each occurrence of adjacency $\bfb\bfa$ by $\bfd$, and each remaining $\bfa$ by $\bfc$. This definition is valid because a Simsun permutation has no double descents. For instance, if $\pi=423516\in SS_6$, then $u_{D(\pi)}=\bfb\bfa\bfa\bfb\bfa$ and so $cd(\pi)=\bfd\bfc\bfd$. The {\bf ab}-index of $\ms_n$ has the expression in terms of Simsun permutations:
\begin{equation}\label{eq:cd}
\Phi_n(\bfc,\bfd)=\sum_{\pi\in SS_n} cd(\pi).
\end{equation}
It is a classical result (cf.~\cite[Theorem~6.3]{Stanley}) that
$$
\Psi_n(\bfa,\bfb)=\Phi_n(\bfa+\bfb,\bfa\bfb+\bfb\bfa).
$$

In~\cite{Chebikin08}, Chebikin studied the {\bf ab}-index of $\ms_n$ with respect to the alternating descent set statistic:
$$
\widehat{\Psi}_n(\bfa,\bfb)=\sum_{\pi\in\ms_n} u_{\widehat{D}(\pi)}
$$
and observed the following relationship.
\begin{proposition}[Chebikin~\cite{Chebikin08}]\label{pro:che}
The polynomial $\widehat{\Psi}_n(\bfa,\bfb)$ can be written as
$$
\widehat{\Psi}_n(\bfa,\bfb)=\widehat{\Phi}_n(\bfa+\bfb,\bfa\bfb+\bfb\bfa),
$$
where $\widehat{\Phi}_n(\bfc,\bfd)=\Phi_n(\bfc,\bfc^2-\bfd)$.
\end{proposition}
\begin{proof}[{\bf Second proof of Theorem~\ref{th:sim}}]  Note that $\widehat{A}_{n}(t)=\widehat{\Phi}_n(1+t,2t)$.  By Proposition~\ref{pro:che}, we have $\widehat{\Phi}_n(\bfc,\bfd)=\Phi_n(\bfc,\bfc^2-\bfd)$ and so every monomial in $\widehat{\Phi}_n(\bfc,\bfd)$ where $\bfd$ appears $k$ times is of the form $(-2t)^{k}(1+t)^{n-1-2k}$ after setting $\bfc=1+t$ and $\bfd=2t$. Therefore, $a(n,k)$ equals the  coefficient $x^k$ in $\widehat{\Phi}_n(1,1+x)$. In view of~\eqref{eq:cd}, $\widehat{\Phi}_n(1,1+x)=R_{n-1}(1+x)$ and the result follows.
\end{proof}
\begin{corollary}
For $n\geq0$, we have
$$
\widehat{A}_{2n+1}(-1)=2^na_{2n+1,n}=E_{2n+1}.
$$
Consequently, the number of down-up Simsun permutations of length $2n$ is $\frac{E_{2n+1}}{2^n}$, which is the $n$-th reduced tangent number appears as~\cite[A002105]{oeis}.
\end{corollary}
\begin{proof}
Setting $t=-1$ in~\eqref{EGF-Anx} gives
$$
\sum_{n\geq 1}\widehat{A}_{n}(-1)\frac{z^n}{n!}=\frac{\tan(2z)\cos(2z)-\cos(2z)+1}{\tan(2z)\cos(2z)+\cos(2z)+1}=\tan(z)=\sum_{n\geq0}\frac{E_{2n+1}z^{2n+1}}{(2n+1)!}.
$$
The result then follows from Theorem~\ref{th:sim}.
\end{proof}
Another interesting consequence of Theorem~\ref{th:sim} is the following new quadratic recursion for $R_n(x)$, the descent polynomials of Simsun permutations.
\begin{corollary}
For $n\geq1$, we have
\begin{equation}\label{qua:simsun}
R_{n+1}(x)=R_n(x)+x\sum_{i=1}^n{n\choose i}R_{i-1}(x)R_{n-i}(x).
\end{equation}
\end{corollary}
\begin{proof} Setting $q=1$ in~\eqref{rec:2} gives
$$
2\widehat{A}_{n+1}(t)=2(1+t)\widehat{A}_n(t)+\sum_{i=1}^{n-1}(1+t^2){n\choose i}\widehat{A}_i(t)\widehat{A}_{n-i}(t).
$$
Plugging~\eqref{eq:a2} into the above recursion yields
$$
2a_{n+1}(x)=2a_n(x)+(1+x)\sum_{i=1}^{n-1}{n\choose i}a_i(x)a_{n-i}(x),
$$
where $x=\frac{-2t}{(1+t)^2}$.  In view of  Theorem~\ref{th:sim}, this is equivalent to
$$
2R_{n+1}(x)=2R_n(x)+x\sum_{i=1}^n{n+1\choose i}R_{i-1}(x)R_{n-i}(x),
$$
which can be  rewritten as~\eqref{qua:simsun}.
\end{proof}


\section{Divisibility properties}\label{Sec:4}
This section deals with the divisibility properties of the alternating descent polynomials, including the proof of Theorem~\ref{conj:altmaj} and a stronger combinatorialization conjecture.
\subsection{Proof of Theorem~\ref{conj:altmaj}}
For convenience, we set
$$
G_n=\prod_{k=1}^{\lfloor\log_2 n\rfloor}\prod_{i=1}^{\lfloor n/2^k \rfloor}(1+q^i).
$$
We begin with a general equivalence regarding divisibility of a polynomial by $G_n$.

 \begin{theorem}\label{equi:Gn}
 For a fixed integer $n\geq1$ and a polynomial $f(q)\in\mathbb{Z}[q]$,
 \begin{equation}
 f(q)\in G_n\mathbb{Z}[q]\Longleftrightarrow f(q)\in (1+q^m)^{\lfloor n/2m\rfloor}\mathbb{Z}[q] \text{ for all $1\leq m\leq \lfloor n/2\rfloor$}.
 \end{equation}
 \end{theorem}

Before the proof of Theorem~\ref{equi:Gn}, we need some preparations. Throughout this section we let $\zeta_n=e^{2\pi i/n}$. For polynomials $f(q),g(q)\in \mathbb{Z}[q]$, we say that $f(q)|g(q)$ if and only if $g(q)/f(q)\in \mathbb{Z}[q]$. For any polynomial $f(z)$, we denote by $\mathrm{ord}(f,z_0)$ the order of vanishing of $f(z)$ at the point $z_0$.

\begin{lemma}\label{lem:QMm:odd}
For any positive integer $m$, $1+q^m$ only has simple roots. For $m,n \geq 1$, $1+q^m|(1+q^n)$ if and only if $n/m$ is an odd integer.
\end{lemma}
\begin{proof}
Note that
\begin{align}\label{simple-id}
1+q^m=\prod\limits_{0\leq r <m}(q-\zeta_{2m}^{2r+1}).
\end{align}
It is obvious that all the roots of $1+q^m$ are simple roots. It is also clear that $\zeta_{2m}$ is a root of $1+q^n$ if and only if $e^{\pi i n/m}=-1$, i.e., $n/m$ is an odd integer. Therefore, $(1+q^m)|(1+q^n)$ only if $n/m$ is an odd integer. Conversely, when $n/m$ is an odd integer, for any $0\leq r<m$, we have $\zeta_{2m}^{(2r+1)n}=-1$, which implies that any root of $1+q^m$ is also a root of $1+q^n$. Therefore, when $n/m$ is odd, we have $(1+q^m)|(1+q^n)$.
\end{proof}
\begin{remark}
When $n/m$ is odd, the fact $(1+q^m)|(1+q^n)$ can also be seen from the identity $(1+q^n)=(1+q^m)(\sum_{j=0}^{n/m-1}(-1)^jq^{mj})$.
\end{remark}

%
%

We are ready for the proof of Theorem~\ref{equi:Gn}.
\begin{proof}[{\bf Proof of Theorem~\ref{equi:Gn}}]
Recall that
\[
G_n=\prod_{k\ge 1}\prod_{i=1}^{\lfloor n/2^k \rfloor }(1+q^i)
=\prod_{i=1}^{\floor{n/2}}(1+q^i)^{r_i},
\]
where
\[
r_i=\abs{\{k\in\mathbb{Z}^+\colon \lfloor n/2^k \rfloor \ge i\}}.
\]

We now evaluate the orders of $G_n$ at its roots. Observe that all roots of $G_n$ are of the form $\zeta_{2m}^{2r+1}$ where $1\leq m \leq n$, $0\leq r <m$ and $(2r+1,m)=1$. For $0\leq r <m$ and $(2r+1,m)=1$, since $\zeta_{2m}^{2r+1}$ is a root of $1+q^i$ if and only if $i/m$ is an odd integer, we deduce that
\begin{align*}
\mathrm{ord}(G_n, {\zeta_{2m}^{2r+1}})&=\sum_{\begin{smallmatrix} lm\leq \lfloor n/2\rfloor \\ l ~~\text{odd} \end{smallmatrix}} r_{lm}
=\sum_{ l ~~\text{odd} } \abs{\{k\geq 1: \lfloor \frac{n}{2^k}\rfloor \geq lm \}}\\
&=\abs{\{s=2^kl: k\geq 1, \text{$l$ odd}, s\leq n/m\}} \\
&=\abs{\{s: \text{$s$ even},s \leq n/m\}}=\left\lfloor \frac{n}{2m}\right\rfloor.
\end{align*}
Therefore, we have
\begin{align}\label{G-prod}
G_n=\prod\limits_{m=1}^n \prod\limits_{\begin{smallmatrix} 0\leq r<m \\ (2r+1,m)=1\end{smallmatrix}} (q-\zeta_{2m}^{2r+1})^{\lfloor n/2m\rfloor}=\prod\limits_{m=1}^n \Phi_{2m}(q)^{\lfloor n/2m\rfloor}.
\end{align}
Here $\Phi_k(x)$ is the cyclotomic polynomial defined by
$$\Phi_k(x)=\prod\limits_{\begin{smallmatrix} 1\leq r <k \\ (r,k)=1\end{smallmatrix}}(x-\zeta_{k}^r).$$
It is well known that $x^n-1=\prod_{d|n}\Phi_d(x)$. We deduce that
\begin{align}\label{1plus}
1+q^m=\frac{1-q^{2m}}{1-q^m}=\prod\limits_{d|m, 2d\nmid m} \Phi_{2d}(q).
\end{align}
Comparing \eqref{G-prod} with \eqref{1plus}, we see that $(1+q^m)^{\lfloor n/2m\rfloor}|G_n$ for any $1\leq m \leq n$.

If $G_n|f(q)$, then it is obvious that $(1+q^m)^{\lfloor n/2m\rfloor}|f(q)$.

Conversely, if $(1+q^m)^{\lfloor n/2m\rfloor}|f(q)$ holds for all $1\leq m \leq \lfloor n/2\rfloor$, then by \eqref{1plus} we have $\Phi_{2m}(q)^{\lfloor n/2m\rfloor}|f(q)$. From \eqref{G-prod} we know that $G_n|f(q)$.
\end{proof}
\begin{remark}
In his study of the divisibility property of
the $q$-tangent numbers, Foata~\cite{Foa81} introduced the polynomial
\[
Ev_k(q)=\prod_{j=0}^l(1+q^{2^jm}),
\]
for any integer $k$ written as $k=2^lm$ with $l\in\mathbb{N}$ and odd $m$.
It was shown \cite[Lemma 2.1]{Foa81} that
\begin{align}\label{Ev-prod}
Ev_k(q)=\prod\limits_{m|k}\Phi_{2m}(q).
\end{align}
Therefore, we have
\begin{align}
\prod_{k=1}^n Ev_k(q)=\prod_{k=1}^n \prod\limits_{m|k}\Phi_{2m}(q)=\prod_{m=1}^n \Phi_{2m}(q)^{\lfloor n/m\rfloor}.
\end{align}
Together with the easy fact $\lfloor \frac{2n}{2m}\rfloor =\lfloor \frac{2n+1}{2m}\rfloor$ and \eqref{G-prod}, we get the following relation:
\begin{align}\label{G-Ev}
G_{2n}=G_{2n+1}=\prod_{k=1}^n Ev_k(q).
\end{align}
\end{remark}

\begin{theorem}\label{div:altmaj}
For any $n,m\ge 1$,
\begin{equation*}
\widehat A_n(1,q)=\sum_{\pi\in\ms_n}q^{\altmaj(\pi)}
\in (1+q^m)^{\lfloor n/2m\rfloor}\mathbb{Z}[q].
\end{equation*}
\end{theorem}

Before we prove Theorem~\ref{div:altmaj}, we show that it is equivalent to Theorem~\ref{conj:altmaj}.

\begin{proof}[{\bf Proof of Theorem~\ref{conj:altmaj}}] In view of  Theorem~\ref{equi:Gn}, Theorem~\ref{div:altmaj} is equivalent to
$$
\widehat A_n(1,q)=\widehat{E}_n(q)G_n
$$
for some $\widehat{E}_n(q)\in\mathbb{Z}[q]$. Considering the complement of permutations,
one sees that the alternating major polynomial $\widehat A_n(1,q)$ is palindromic. Thus, the palindromicity of $\widehat{E}_n(q)$ follows from the following  basic properties.
\begin{enumerate}
\item
Let $f(q),g(q)\in\mathbb{Z}[q]$ such that $f(q)$ and $g(q)$ are both palindromic.
Then the product $f(q)g(q)$ is palindromic.
\item
Let $f(q),g(q),h(q)\in\mathbb{Z}[q]$ such that $f(q)=g(q)h(q)$.
Suppose that $g(q)$ is palindromic.
Then,
\[
\text{$f(q)$ is palindromic}
\iff
\text{$h(q)$ is palindromic}.
\]
\end{enumerate}
These two properties are easy consequences of the fact that a polynomial $f(q)\in\mathbb{Z}[q]$
of degree $n$ is palindromic
if and only if $f(q)=f(1/q)q^n$.
\end{proof}

It remains to prove Theorem~\ref{div:altmaj}. As usual, for integer $n\geq0$, let $(q;q)_n:=\prod_{i=1}^n(1-q^i)$.
\begin{lemma}\label{lem-order}
For any integer $n,m\geq 1$, the order of $1+q^m$ in $(q;q)_n$ is $\lfloor n/2m\rfloor$. That is,
$$(1+q^m)^{\lfloor n/2m\rfloor} |(q;q)_n, \quad \text{and} \quad (1+q^m)^{\lfloor n/2m\rfloor+1} \nmid (q;q)_n.$$
\end{lemma}
\begin{proof}
Recall \eqref{simple-id}. Given $0\leq r<m$, $\zeta_{2m}^{2r+1}$ is a root of $1-q^k$ if and only if $2m|(2r+1)k$, and in this case,  $\zeta_{2m}^{2r+1}$ is a simple root of $1-q^k$. In particular, $\zeta_{2m}$ is a root of $1-q^k$ if and only if $2m|k$. Therefore, $\mathrm{ord}((q;q)_n, \zeta_{2m})=\lfloor n/2m\rfloor$. It is easy to see that this order is the same as the order of $1+q^m$ in $(q;q)_n$.
\end{proof}

We are in position to prove Theorem~\ref{div:altmaj}.

\begin{proof}[{\bf Proof of Theorem~\ref{div:altmaj}}]
Multiplying both sides of~\eqref{gen:rem} by $1-t$ and letting $t$ tend to  $1$ gives
\begin{equation*}
\sum_{n=0}^{\infty}\frac{z^n}{n!}\frac{\widehat{A}_n(1,q)}{(q;q)_{n}}=\prod_{j=0}^{\infty}(\sec(zq^j)+\tan(zq^j))=:F(z).
\end{equation*}
%
%
By the definition of $F(z)$, it is easy to see that $F(0)=1$. Moreover, we have
\begin{align*}
\widehat{A}_n(1,q)=F^{(n)}(0)(q;q)_n.
\end{align*}
Taking logarithmic differentiation, and noting that
$$(\sec x)'=\sec x \tan x, \quad  (\tan x)'=\sec^2 x,$$
we get
\begin{align}\label{1st-d}
\frac{F'(z)}{F(z)}=\sum_{j=0}^\infty \sec(zq^j)q^j.
\end{align}
It follows that $F'(0)=\frac{1}{1-q}$. Recall that
\begin{align}
\sec x=\sum_{n=0}^\infty \frac{(-1)^nE_{2n}}{(2n)!}x^{2n},
\end{align}
where $E_k$ is the $k$-th Euler number. We have
\begin{align}\label{sec-eval}
\sec^{(n)}(0)=\left\{\begin{array}{ll} 0 &\text{$n$ is odd}, \\
(-1)^k E_{2k} & \text{$n=2k$ is even}.
\end{array}\right.
\end{align}

By \eqref{1st-d} we obtain
\begin{align}
\frac{d^{n-1}}{dx^{n-1}} \left(\frac{F'(x)}{F(x)} \right)\mid_{x=0}&=\sum_{j=0}^\infty \sec^{(n-1)}(0)q^{jn}=\frac{\sec^{(n-1)}(0)}{1-q^{n}}. \label{n-derivative}
\end{align}
By Fa\`{a} di Bruno's formula, for any $n$-times differentiable function $f(x)$,  we have
\begin{align}\label{FdB}
\frac{d^n}{dx^n} \ln f(x)=\sum_{\begin{smallmatrix}m_1+2m_2+\cdots +nm_n=n \\ m_i \geq 0\end{smallmatrix}} & \frac{n!}{m_1!m_2!\cdots m_n!}\frac{(-1)^{m_1+\cdots +m_n-1}(m_1+\cdots +m_n-1)!}{f(x)^{m_1+\cdots +m_n}}\nonumber \\
&\prod\limits_{1\leq j \leq n}\left(\frac{f^{(j)}(x)}{j!} \right)^{m_j}.
\end{align}
Setting $f(x)=F(x)$. By \eqref{n-derivative} and \eqref{FdB} we get
\begin{align}\label{Fn0}
F^{(n)}(0)=-\sum_{\begin{smallmatrix}m_1+2m_2+\cdots +(n-1)m_{n-1}=n \\ m_i \geq 0\end{smallmatrix}}& \frac{n!(m_1+\cdots +m_{n-1}-1)!}{m_1!m_2!\cdots m_{n-1}!}(-1)^{m_1+\cdots +m_{n-1}-1}\nonumber \\
&\prod\limits_{1\leq j \leq n-1}\left(\frac{F^{(j)}(0)}{j!} \right)^{m_j}+\frac{\sec^{(n-1)}(0)}{1-q^{n}}.
\end{align}

Given any integer $m\geq 1$, we claim that for any integer $r$, $F^{(r)}(0)\in P_m$, where
$$P_m=\left\{\frac{f(q)}{g(q)}: f(q), g(q)\in \mathbb{Z}[q], (f(q),g(q))=1 \text{ and } (1+q^m,g(q))=1\right\}.$$
This is obvious for $r=1$. Now suppose this claim holds for $r<n$. We are going to show that it also holds for $r=n$.

If $n=2k+1$ is odd, then note that $(1+q^m,1-q^{2k+1})=1$, from \eqref{Fn0} we see that $F^{(n)}(0)\in P_m$.

If $n$ is even, then by \eqref{sec-eval} we have $\sec^{(n-1)}(0)=0$. Thus $F^{(n)}(0)\in P_m$ follows directly from \eqref{Fn0} and induction hypothesis.

By mathematical induction, we see that the claim indeed holds for all $r$.

Now note that $\widehat{A}_n(1,q)=F^{(n)}(0)(q;q)_n$. By the claim and Lemma \ref{lem-order}, we see that the order of $1+q^m$ in $\widehat{A}_n(1,q)$ is at least $\lfloor n /2m\rfloor$.
\end{proof}

\subsection{Divisibility of $\widehat{A}_n(q^{j},q)$ via  the quadratic recursion}

The special $m=1$ case of Theorem~\ref{div:altmaj} asserts that $\widehat{A}_n(1,q)$ is divisible by $(1+q)^{\lfloor n/2\rfloor}$ over $\mathbb{Z}(q)$. This can be generalized to the following divisibility of $\widehat{A}_n(q^{j},q)$ using the quadratic recursion~\eqref{rec:2}.

\begin{theorem}
For $n\geq1$ and $j\geq0$,
\begin{itemize}
\item when $n$ is even, then $\widehat{A}_n(q^{j},q)$ is divisible by $(1+q)^{\lfloor n/2\rfloor}$ (resp.~$(1+q)^{\lfloor (n-1)/2\rfloor}$) if $j$ is even (resp.~odd);
\item when $n$ is odd, then   $\widehat{A}_n(q^{j},q)$ is divisible by $(1+q)^{\lfloor n/2\rfloor}$ for all $j\geq0$.
\end{itemize}
In particular, $\widehat{A}_n(1,q)$ is divisible by $(1+q)^{\lfloor n/2\rfloor}$ over $\mathbb{Z}(q)$.
\end{theorem}
\begin{proof}
The two statements are true for $n=1,2$, as $A_1(t,q)=1$ and $A_2(t,q)=1+tq$.
Setting $t=q^j$ in~\eqref{rec:2} gives
\begin{multline}\label{rec:altmaj}
2\widehat{A}_{n+1}(q^j,q)=(1+q^{j+1})\widehat{A}_n(q^{j+1},q)+(1+q^{n+j})\widehat{A}_n(q^j,q)\\
+\sum_{i=1}^{n-1}(1+q^{2(i+j)+1}){n\choose i}\widehat{A}_i(q^j,q)\widehat{A}_{n-i}(q^{i+j+1},q)
\end{multline}
for $n\geq1$.
Based on this quadratic recursion,  it is routine to check the two statements by induction on $n$.
\end{proof}

\subsection{A stronger combinatorialization conjecture}

We shall provide  a combinatorial way to show that $(1+q^m)^r$
divides a polynomial over $\mathbb{Z}[q]$ and propose a combinatorialization conjecture which is stronger than Theorem~\ref{div:altmaj}.

\begin{lemma}[Reduction lemma]\label{lem:reduction}
Fix $m\in\mathbb{Z}^+$.
Let $f,g\colon X\to \mathbb{Z}$ be functions defined on a finite set $X$.
For any $k\in\mathbb{Z}$, define a function $W_k\colon \mathbb{Z}_{2m}\to 2^X$ by
\[
W_k(l)=\{x\in X\colon kf(x)\equiv l\!\!\!\!\pmod{2m}\}.
\]
Then the statement
\[
\sum_{x\in W_k(l)}g(x)
=\sum_{x\in W_k(l+m)}g(x)
\qquad\text{for any $l\in\mathbb{Z}_m$}
\]
holds for all odd integer $k$ if it holds for $k=1$.
\end{lemma}
\begin{proof}
Let $k$ be an odd integer.
Let $d=\gcd(k,m)$ be the greatest common divisor of $k$ and~$m$.
Then $d$ is odd.
Denote $k'=k/d$ and $m'=m/d$.
Let $l\in\mathbb{Z}_m$.
If $d$ does not divide~$l$,
then $W_k(l)=W_k(l+m)=\emptyset$ and we have nothing to show.
Below we can suppose that~$d$ divides $l$.
Then $l'=l/d\in\mathbb{Z}$.

From definition, one infers that $\gcd(k',\,2m')=1$,
and there is an odd integer~$h$ such that $k'h\equiv 1\pmod{2m'}$.
Solving the equation $kf(x)\equiv l\pmod{2m}$, we find
\begin{align*}
W_k(l)
&=\{x\in X\colon f(x)\equiv hl'\!\!\!\!\pmod{2m'}\}
=\sqcup_{s\in\mathbb{Z}_d}W_1(hl'+2m's),
\end{align*}
where the symbol $\sqcup$ denotes disjoint union.
Replacing $l$ by $l+m$, we deduce that
\begin{align}
W_k(l+m)
&=\bigsqcup_{s\in\mathbb{Z}_d} W_1\brk1{h(l'+m')+2m's}\notag\\
&=\bigsqcup_{s\in\mathbb{Z}_d} W_1\brk2{h(l'+m')+2m'\brk1{s+(d-h)/2}}\label{pf1}\\
&=\bigsqcup_{s\in\mathbb{Z}_d} W_1\brk1{hl'+2m's+m},\notag
\end{align}
where Eq.~\eqref{pf1} holds since the integer $s+(d-h)/2$ runs over $\mathbb{Z}_d$
as $s$ does. Since the desired statement holds for $k=1$, we can derive that
\begin{align*}
\sum_{x\in W_k(l+m)}g(x)
=\sum_{s\in\mathbb{Z}_d}\sum_{x\in W_1(hl'+2m's+m)}g(x)
=\sum_{s\in\mathbb{Z}_d}\sum_{x\in W_1(hl'+2m's)}g(x)
=\sum_{x\in W_k(l)}g(x).
\end{align*}
This completes the proof.
\end{proof}

Below is a combinatorial way to show that $(1+q^m)^r$
divides an enumerative  polynomial over $\mathbb{Z}[q]$.

\begin{proposition}\label{cor:Qm}
Fix $m\in\mathbb{Z}^+$ and $r\in\mathbb{N}$.
Let $\st\colon X\to \mathbb{Z}$ be a function defined on a finite set~$X$.
Then $\sum_{x\in X}q^{\st(x)}\in (1+q^m)^r\mathbb{Z}[q]$
if
\[
\sum_{\st(x)\equiv l\!\! \!\!\pmod{2m}}\binom{\st(x)}{j}
=\sum_{\st(x)\equiv l+m\!\! \!\!\pmod{2m}}\binom{\st(x)}{j}.
\]
for any $l\in \mathbb{Z}_m$ and any $0\le j\le r-1$.
\end{proposition}

\begin{proof}
The zero set of the polynomial $1+q^m$ is
\[
\brk[c]1{e^{\pm k\pi i/m}\colon\text{$1\le k\le m$ and $k$ is odd}}.
\]
For $j\ge 0$,
we denote the $j$th derivative of the polynomial $\sum_{x\in X}q^{\st(x)}$ by $D_j(q)$. Then
\[
D_j(q)
=\frac{j!}{q^j}\sum_{x\in X}\binom{\st(x)}{j}q^{\st(x)}.
\]
Since the number $D_j\brk1{e^{-k\pi i/m}}$
is the conjugate of $D_j\brk1{e^{k\pi i/m}}$,
we find
\begin{equation}\label{div:st}
\sum_{x\in X}q^{\st(x)}\in (1+q^m)^r\mathbb{Z}[q]
\end{equation}
if and only if
$D_j\brk1{e^{k\pi i/m}}=0$
for any odd $1\le k\le m$ and any $0\le j\le r-1$. Set
$$
c_j(x)=j!{\st(x)\choose j}.
$$
Since
\[
D_j\brk1{e^{k\pi i/m}}
=\sum_{x\in X}c_j(x)\exp\brk3{\frac{k(\st(x)-j)\pi i}{m}},
\]
the divisibility in~\eqref{div:st} holds if
\[
\sum_{k(\st(x)-j)\equiv l\!\! \!\!\pmod{2m}}c_j(x)
=\sum_{k(\st(x)-j)\equiv l+m\!\! \!\!\pmod{2m}}c_j(x)
\quad\text{for all $l\in\mathbb{Z}_m$},
\]
that is,
\[
\sum_{k\times\st(x)\equiv l\!\! \!\!\pmod{2m}}c_j(x)
=\sum_{k\times\st(x)\equiv l+m\!\! \!\!\pmod{2m}}c_j(x)
\quad\text{for all $l\in\mathbb{Z}_m$}.
\]
By Lemma~\ref{lem:reduction}, this is equivalent to the premise.
\end{proof}

In view of Proposition~\ref{cor:Qm},  we propose the following conjecture which is stronger than Theorem~\ref{div:altmaj}.
\begin{conjecture}\label{conj:altmaj2}
Let $n\ge 1$ and $1\le m\le \floor{n/2}$.
For any $l\in \mathbb{Z}_m$ and any $0\le j\le \floor{n/2m}-1$,
\[
\sum_{\pi\in\ms_n\atop\altmaj(\pi)\equiv l\!\! \!\!\pmod{2m}}\binom{\altmaj(\pi)}{j}
=\sum_{\pi\in\ms_n\atop\altmaj(\pi)\equiv l+m\!\! \!\!\pmod{2m}}\binom{\altmaj(\pi)}{j}.
\]
Consequently,
$$
\sum_{\pi\in\ms_n}q^{\altmaj(\pi)}
\in (1+q^m)^{\lfloor n/2m\rfloor}\mathbb{Z}[q].
$$
\end{conjecture}

This conjecture has been verified for $n\leq 20$. At the beginning, it was our attempt to prove Theorem~\ref{div:altmaj} by verifying    Conjecture~\ref{conj:altmaj2}. Unfortunately, we  can only prove  Conjecture~\ref{conj:altmaj2} for the special  $j=0$ case.

\begin{theorem}
Conjecture~\ref{conj:altmaj2} is true for $j=0$. Thus, $(1+q^m)$ divides $\sum_{\pi\in\ms_n}q^{\altmaj(\pi)}$ for $1\le m\le \floor{n/2}$.
\end{theorem}

\begin{proof}
Let
$$
\bs_{n,l}^{(m)}:=\{\pi\in\bs_n:\altmaj(\pi)\equiv l\!\! \!\!\pmod{2m}\}.
$$
Then, the result is equivalent to
$$
|\bs_{n,l}^{(m)}|=|\bs_{n,m+l}^{(m)}|
$$
for each $0\leq l\leq m-1$. For each $\pi=\pi_1\pi_2\ldots\pi_n\in\bs_{n,l}^{(m)}$, define $\pi'=\pi_1'\pi_2'\ldots\pi_n'$, where
$$
\pi'_i=
\begin{cases}
\pi_i,\quad&\text{if $2m<i\leq n$};\\
\pi_{2m+1-i},\quad&\text{if $1\leq i\leq2m$}.
\end{cases}
$$
This mapping sets up a bijection between $\bs_{n,l}^{(m)}$ and $\bs_{n,m+l}^{(m)}$.
For example, if  $\pi=\blue{942357}861\in\bs_{9,0}^{(3)}$, then $\pi'=\red{753249}861$. It is easy to verify that
$$
\Altdes(\pi)=\{1,4,6,7\}\quad{and}\quad\Altdes(\pi')=\{1,3,4,7\},
$$
and so $\pi'\in\bs_{9,3}^{(3)}$.
\end{proof}
We hope that the above proof of the $j=0$ case would shed some light on finding a bijective proof of Conjecture~\ref{conj:altmaj2}. It would be interesting to see whether  Proposition~\ref{cor:Qm} can be applied to establish any known or new divisibility properties of combinatorial polynomials. For instance, can Proposition~\ref{cor:Qm} be used to prove combinatorially the divisibility properties of the {\em$q$-tangent numbers} found two decades ago by Andrews--Gessel~\cite{AG} and Foata~\cite{Foa81}?

\section{Concluding remarks and further conjectures}
\label{Sec:final}

Theorem~\ref{main:1} provides a class of combinatorial polynomials which are unimodal but have $\gamma$-vector alternates in sign. Note that combinatorial polynomials with such kind of phenomena have already been reported recently by Brittenham, Carroll, Petersen and   Thomas~\cite{BCPT} and by Sagan and Tirrell~\cite{SJ}.  It might  be interesting to investigate systematically polynomials arising in enumerative combinatorics whose gamma expansions have coefficients alternate in sign.

Theorem~\ref{main:1}  may be generalized in  two directions as follows.

\subsection{A log-concavity conjecture}
Another property that is stronger than unimodality is the so-called  log-concavity. Recall that a polynomial $h(t)=\sum_{k=0}^nh_kt^k\in\mathbb{R}[t]$ is said to be {\em log-concave} if $h_i^2\geq h_{i-1}h_{i+1}$ for all $1\leq i\leq n-1$. If the coefficients of $h(t)$ have {\em no internal zero}, that is, there do not exist integers $0\leq i<j<k\leq n$ such that $a_i\neq0,a_j=0,a_k\neq0$, then the log-concavity of $h(t)$ implies its unimodality. Based on calculations and Theorem~\ref{unimodal}, we posed the following conjecture.

\begin{conjecture}
The alternating Eulerian polynomial $\widehat{A}_n(t)$ is log-concave for any $n\geq1$.
\end{conjecture}

This conjecture has been verified for $n\leq 1000$.

\subsection{Two further $\gamma$-positivity conjectures}

\begin{conjecture}\label{gam:q}
The polynomial $\widehat{A}_n(t,q)$ has the $q$-gamma expansion
\begin{equation}\label{exp:q}
\sum_{\pi\in \ms_n}t^{\altdes(\pi)}q^{\altmaj(\pi)}=\sum_{k=0}^{\lfloor(n-1)/2\rfloor}\widehat{\gamma}_{n,k}(q)q^{{k+1\choose 2}}(-t)^k\prod_{i=k+1}^{n-1-k}(1+tq^i),
\end{equation}
where $\widehat{\gamma}_{n,k}(q)\in\mathbb{N}[q]$ and has $(1+q)^k$ as a divisor.
\end{conjecture}

Conjecture~\ref{gam:q} is an analog of the $q$-$\gamma$-positivity expansion due to Han, Jouhet and Zeng~\cite{hjz} for the {\em Euler--Mahonian polynomials} $\sum_{\pi\in\ms_n}t^{\des(\pi)}q^{\maj(\pi)}$.
For each $\pi\in\ms_n$, let $\mathcal{R}(\pi):=\pi_n\pi_{n-1}\cdots\pi_1$ be its {\em reversal}. When $n$ is even (resp.~odd), the mapping $\Theta=\mathcal{C}\circ\mathcal{R}$ (resp.~$\Theta=\mathcal{R}$) is an involution on $\ms_n$ such that if $\pi'=\Theta(\pi)$, then
$$
\altdes(\pi')=n-1-\altdes(\pi)\quad\text{and}\quad\altmaj(\pi')={n\choose2}-n\times\altdes(\pi)+\altmaj(\pi).
$$
The existence of expansion~\eqref{exp:q} with $\widehat{\gamma}_{n,k}(q)\in\mathbb{Z}[q]$ then follows from $\Theta$ and~\cite[Lemma~4.3]{Lin}.
The first few values of $\widehat{\gamma}_{n,k}(q)$ are listed in Table~\ref{tq:gamma}.

\begin{table}
\begin{tabular}{|p{0.6in}|c|c|c|c|c|cl} \hline
\centering $\widehat{\gamma}_{n,k}(q)$ & $k=0$ &$k=1$&$k=2$&$k=3$\\
\hline
\centering $n=2$ & $1$ &&& \\
\hline
\centering $n=3$ & $2$ & $1+q$&&\\
\hline
\centering $n=4$ & $5$& $2(1+q)^2$&&\\
\hline
\centering $n=5$ & $16$& $(1+q)(7+5q+7q^2)$&$(1+q)^2(2+2q^2)$& \\
\hline
\centering $n=6$ & $61$& $(1+q)^2(26-5q+26q^2)$&$(1+q)^2(1+q^2)\times\atop\qquad\quad(5+7q+5q^2)$& \\
\hline
\centering $n=7$ & $272$& $(1+q)(117+91q+103q^2\atop+91q^3+117q^4)$&$(1+q)^2(1+q^2)(1+q+q^2)\times\atop\qquad\qquad(26-5q+26q^2)$&$(1+q)^2(1+q^2)(1+q^3)\times\atop\qquad\quad\quad(12-7q+12q^2)$\\
\hline
\centering $n=8$ & $1385$& $6(1+q)^2(99-21q+106q^2\atop-21q^3+99q^4)$&$2(1+q)^2(1+q^2)(63+62q+98q^2\atop+118q^3+98q^4+62q^5+63q^6)$&$(1+q)^3(1+q^2)(1+q^3)\times\atop(21-14q+48q^2-14q^3+21q^4)$\\
\hline
\end{tabular}
\vspace{.2in}
\caption{The values of $\widehat{\gamma}_{n,k}(q)$ for $2\leq n\leq8$}
\label{tq:gamma}
\end{table}

It was conjectured  in 2005 by Gessel  (see~\cite[p.~446]{bran}) and latter confirmed by the first author~\cite{Lin0} that the {\em two-sided Eulerian polynomials} $\sum_{\pi\in\ms_n}s^{\des(\pi^{-1})}t^{\des(\pi)}$ can be expanded in the basis (of polynomials with certain two-sided symmetries)
$$
\{(st)^i(1+st)^j(s+t)^{n-1-2i}\}
$$
with nonnegative coefficients. This refines the $\gamma$-positivity of the Eulerian polynomials and inspires the following fascinating refined two-sided $\gamma$-expansion.

\begin{conjecture}\label{gam:dou} Introduce the {\em two-sided alternating Eulerian polynomials}
$$
\widetilde{A}_n(s,t):=\sum_{\pi\in\ms_n}s^{\altdes(\pi^{-1})}t^{\altdes(\pi)}.
$$
Then, for $n\geq1$
\begin{equation}\label{exp:dou}
\widetilde{A}_n(s,t)=\sum_{i,j\geq0\atop j+2i\leq n-1}\widehat{\gamma}_{n,i,j}(-st)^i(1+st)^j(s+t)^{n-1-j-2i},
\end{equation}
where $\widehat{\gamma}_{n,i,j}$ are nonnegative integers.
\end{conjecture}

The existence of expansion~\eqref{exp:dou} is guaranteed by~\cite[Lemma~5]{Lin0} and the required symmetries of $\widetilde{A}_n(s,t)$ that follow from the complement and inverse of permutations. We list the first few expansions of $\widetilde{A}_n(s,t)$ here:
\begin{align*}
\widetilde{A}_2(s,t)&=1+st,\\
\widetilde{A}_3(s,t)&=(1+st)^2+(s+t)^2-2st,\\
\widetilde{A}_4(s,t)&=2(1+st)^3+(s+t)^3+2(s+t)^2(1+st)-5st(1+st)-3st(s+t),\\
\widetilde{A}_5(s,t)&=3(1+st)^4+2(1+st)^3(s+t)+6(1+st)^2(s+t)^2+2(1+st)(s+t)^3\\
&\quad+3(s+t)^4-14st(1+st)-10st(1+st)(s+t)-14st(s+t)^2
+16(st)^2.
\end{align*}

It would be interesting to compute the generating function for the two-sided alternating Eulerian polynomials $\widetilde{A}_n(s,t)$.

Conjectures~\ref{gam:q}  and~\ref{gam:dou} have been verified for $n\leq10$.

\section*{Acknowledgement}
The authors thank  Yuan-Hsun Lo for his helpful discussions.
The first author  was supported by the National Natural Science Foundation of China grant 11871247 and the project of Qilu Young Scholars of Shandong University. The second author was supported by the National Natural Science Foundation of China grant 12071063. The third author was supported by  the National Natural Science Foundation of China grant 11671037. The fourth author was supported  by the National Natural Science Foundation of China grant 11801424 and a start-up research grant of Wuhan University.

\end{document}